\documentclass[a4paper,11pt]{article}
\usepackage{amsmath,amsthm,amssymb,amsfonts,graphicx}
\usepackage[top=3cm, bottom=3cm, left=3cm, right=3cm]{geometry}
\usepackage[shortlabels]{enumitem}
\usepackage{booktabs}

\usepackage{hyperref}
\usepackage{bm}

\usepackage{subcaption}

\usepackage{tikz}
\usepackage{tkz-graph}
\tikzstyle{vertex}=[circle, draw, inner sep=0pt, minimum size=6pt]

\usetikzlibrary{shapes.geometric}
\usetikzlibrary{svg.path}

\newtheorem{defin}{Definition}[section]

\newtheorem{theorem}[defin]{Theorem}
\newtheorem{remark}[defin]{Remark}
\newtheorem{observation}[defin]{Observation}
\newtheorem{corollary}[defin]{Corollary}
\newtheorem{lemma}[defin]{Lemma}
\newtheorem{example}[defin]{Example}

\DeclareMathOperator{\rank}{rank}

\title{Complete multipartite graphs that are determined, up to switching, by their Seidel spectrum\thanks{The work was supported by the joint ISF-NSFC joint scientific research program, jointly
funded
by the National Natural Science Foundation of China and the Israel Science Foundation (grant nos. 11561141001,  \ 2219/15), and by the National Natural Science Foundation of
China
grant  no. 11531001.}}

\author{Abraham Berman\thanks{Technion-Israel Institute of Technology, Haifa 32000, Israel. Email: \texttt{berman@technion.ac.il}}\\ Naomi Shaked-Monderer\thanks{The Max Stern Yezreel Valley College, Yezreel Valley 19300, Israel. Email: \texttt{nomi@technion.ac.il}}\\ Ranveer Singh\thanks{ Technion-Israel Institute of Technology, Haifa 32000, Israel. Email: \texttt{ranveer@iitj.ac.in}}\\
  Xiao-Dong Zhang \thanks{Shanghai Jiao Tong University,
800 Dongchuan road, Shanghai, 200240,
P.R. China.\ \ Email: \texttt{xiaodong@sjtu.edu.cn} (Corresponding author)}
}

\date{\today}

\begin{document}
\maketitle
\begin{abstract}
\noindent

It is known that complete multipartite graphs are determined by their distance spectrum but not by their adjacency spectrum.
The Seidel spectrum of a graph $G$ on more than one vertex does not determine the graph, since any graph obtained from $G$ by Seidel switching has the same Seidel spectrum. We consider $G$ to be determined by its Seidel spectrum, up to switching, if any graph with the same spectrum is switching equivalent to a graph isomorphic to $G$.
It is shown that any graph which has the same spectrum as a complete $k$-partite graph is switching equivalent to a complete $k$-partite graph, and if the different partition sets sizes are $p_1,\ldots, p_l$, and there are at least three partition sets of each size $p_i$, $i=1,\ldots, l$, then $G$ is determined, up to switching, by its Seidel spectrum.  Sufficient conditions for a complete tripartite graph to be determined by its Seidel spectrum are discussed, and a conjecture is made on complete tripartite graphs on more than 18 vertices.

\medskip

\noindent
\textbf{Keywords:} Complete multipartite graphs, Seidel matrix, Seidel switching, $S$-determined
\smallskip
\noindent
\textbf{Mathematical Subject Classification 2010:} 05C50, 05C12
\end{abstract}

\section{Introduction}
The graphs in this paper are simple. Let $V(G)$ denote the vertex set of a graph $G$ and let $E(G)$ denote its edge set. If two distinct vertices $v_i, v_j \in V(G)$ are adjacent in $G$, we write $v_i\sim v_j$, otherwise, $v_i \not \sim v_j$. The \emph{adjacency matrix}, $A(G)=(a_{ij})$, of a graph $G$ with $V(G)=\{v_1,\ldots,v_n\}$ is an $n\times n$ $(0,1)$ symmetric matrix in which $a_{ij}=1$ if and only if $v_i \sim v_j$. The \emph{spectrum} of $G$  is the multiset of the eigenvalues of $A(G)$.  Several parameters of a graph can be deduced from the spectrum of $A(G)$. For example,   $\text{trace}(A(G)^k)=\sum_{i=1}^n\lambda_i^k$ is the number of all closed walks, of length $k$ in $G$, so in particular, $|E(G)|=\frac{1}{2}\text{trace}(A(G)^2)$ and the number of triangles in $G$ is $\frac{1}{6}\text{trace}(A(G)^3)$. Thus it is interesting to know which graphs are determined (up to isomorphism) by their spectrum, that is, graphs for which there exist no   non-isomorphic graph with the same spectrum. Two non-isomorphic graphs $G$ and $H$ are \emph{cospectral} if they have the same spectrum. A classical example of non-isomorphic cospectral graphs is given in Figure \ref{nnn} \cite{cvetkovic1971graphs}; their common spectrum is $\{[2]^1, [0]^3, [-2]^1\}$ (exponents indicate multiplicities). The complete graph $K_n$ and the path graph $P_n$ are determined by their spectrum.

Let $V(G)=\{v_1,\ldots,v_n\}$. Then $D(G)=(d_{ij})$ is the diagonal matrix with $d_{ii}$ the degree of $v_i$. Let $I$ denote the identity matrix and $J$ the all-ones matrix. A linear combination of $A(G), D(G), J$ and $I$ is called a \emph{generalised adjacency matrix}. There are many results on the spectra of generalised adjacency matrices, see the excellent surveys \cite{van2003graphs,haemers2004enumeration,vanDamHaemers2009}.  Generalised adjacency matrices include the {\it Laplacian}, $L(G)=D(G)-A(G)$, the {\it signless Laplacian}, $Q(G)=D(G)+A(G)$, and the {\it Seidel matrix} $S(G)=J-I-2A(G)$. Note that  the Seidel matrix $S(G)= (s_{ij})$ of $G$ is the square matrix of order $n$ defined by
\[s_{ij}=\begin{cases}
0 & \mbox{if $v_i=v_j$},\\
1 & \mbox{if $v_i\not \sim v_j$, $v_i\neq v_j$},\\
-1 & \mbox{if $v_i\sim v_j$, $v_i\neq v_j$.}
\end{cases}\]  Other matrices for which the spectrum is of interest are the \emph{distance matrix}, where the $(i,j)$ entry is the distance between $v_i$ and $v_j$, and the
 {\it normalized Laplacian}, $\mathcal{L}(G)=D(G)^{-\frac{1}{2}}L(G)D(G)^{-\frac{1}{2}}$. Let $X\in \{$generalised adjacency, Laplacian, signless Laplacian, normalised Laplacian, distance, Seidel$\}$. The \emph{$X$ spectrum} of $G$ is the spectrum of the $X$ matrix of $G,$ and the references mentioned above contain many results on finding non-isomorphic $X$ cospectral graphs (i.e., non-isomorphic graphs that have the same $X$ spectrum) or showing that a graph is determined by its $X$ spectrum. In addition, some graphs that are determined by the normalized Laplacian spectrum are given in \cite{butler2016cospectral,berman2018family}, and the references there. Our paper is a small contribution to the rich literature on graphs that are determined by their $X$ spectrum. This is done by considering the Seidel spectrum of complete multipartite graphs. We mention in passing, that complete multipartite graphs are determined by the spectrum of the distance matrix but not by the spectrum of the adjacency matrix \cite{Delorme2012,jin2014complete}.

 Let $U,W\subseteq V(G)$ form a partition of $V(G)$. A \emph{Seidel switching} with respect to $U$ transforms $G$ to a graph $H$ by deleting the edges between $U$ and $W$ and adding an edge between vertices $u\in U$ and $w\in W$ if $(u,w)\notin E(G)$. For more details on Seidel matrices and related topics, see \cite{van1991equilateral,sciriha2018two,szollHosi2018enumeration,balla2018equiangular} and the references there. Seidel switching is an equivalence relation and we say that $G$ and $H$ are \emph{switching equivalent}. In general, $G$ and $H$ are not isomorphic, but since $S(H)=\Lambda S(G) \Lambda$, where $\Lambda$ is a signature matrix (a diagonal matrix with $1$  entries corresponding to vertices of $U$ and $-1$  entries corresponding to vertices of $W$), $S(H)$ and $S(G)$ are similar and have the same spectrum. Hence such $G$ and $H$ are cospectral, so no graph with more than one vertex is determined by its Seidel spectrum.  Hence we say that a graph $G$ is \emph{Seidel determined, up to switching,} (or, in short, {\it $S$-determined}) if the only graphs with same Seidel spectrum are switching equivalent to a graph isomorphic to $G$.

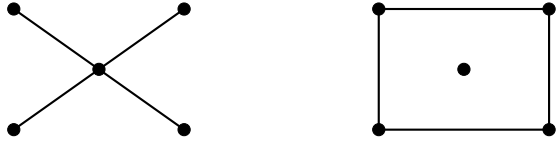
\begin{figure}[!ht]
\begin{center}
\begin{tikzpicture}[scale=0.8]

\draw[thick] (0,2)--(2.8,0);
\draw[thick] (0,0)--(2.8,2);

\draw[fill](0,0) circle[radius=0.1];
\draw[fill](0,2) circle[radius=0.1];
\draw[fill](2.8,0) circle[radius=0.1];
\draw[fill](2.8,2) circle[radius=0.1];
\draw[fill](1.4,1) circle[radius=0.1];

\draw[thick] (6,0)--(6,2)--(8.8,2)--(8.8,0)--cycle;

\draw[fill](6,0) circle[radius=0.1];
\draw[fill](6,2) circle[radius=0.1];
\draw[fill](8.8,0) circle[radius=0.1];
\draw[fill](8.8,2) circle[radius=0.1];
\draw[fill](7.4,1) circle[radius=0.1];

\end{tikzpicture}
\caption{A pair of on-isomorphic cospectral graphs   \label{nnn}}
\end{center}\vspace{-8pt}
\end{figure}

The paper is organized as  follows. In Section \ref{no&pri} we discuss some notations and preliminaries used in the main results. In Section \ref{comultipar} we consider   complete multipartite graphs. We show that each graph  Seidel cospectral with a complete $k$-partite graph is switching equivalent to a complete $k$-partite graph, and describe  cases when   complete multipartite graphs are $S$-determined. Complete tripartite graphs are considered in Section \ref{comtripar}, where we present triples $p,q,r$ for which $K_{p,q,r}$ is $S$-determined and examples of non-isomorphic, non switching equivalent, Seidel cospectral complete tripartite graphs. We conclude with a conjecture on complete tripartite graphs on more than 18 vertices.

\section{Notations and preliminaries}\label{no&pri}

We begin with presenting the Seidel spectrum of some graphs.

\begin{lemma}\label{spectra1}
\begin{enumerate}
\item[{\rm(a)}] The  Seidel spectrum of the empty graph on $n$ vertices is $\{[n-1]^1,[-1]^{n-1} \}$.
\item[{\rm(b)}] The Seidel spectrum of the complete graph $K_n$ is $\{[1]^{n-1}, [1-n]^1 \}$.
\item[{\rm(c)}] The Seidel spectrum of the complete bipartite graph $K_{p,q}$ is $\{[p+q-1]^1,[-1]^{p+q-1}\} $.
\end{enumerate}
\end{lemma}

\begin{proof}
The Seidel matrix of the empty graph on $n$ vertices is the $n\times n$ matrix $J-I$. The Seidel matrix of $K_n$ is $I-J$,  and the graph
 $K_{p,q}$ is   obtained by switching from the empty graph on ${p+q}$ vertices, where $U$ is any subset of $p$ vertices.
\end{proof}

To prove our main theorem we need the  Seidel spectra of some graphs on $5$ vertices with an isolated vertex.

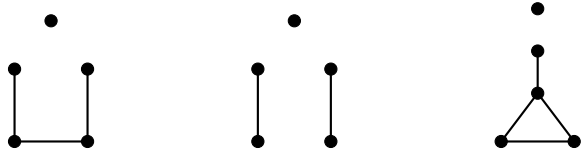
\begin{figure}[!ht]
\begin{center}
\begin{tikzpicture}[scale=0.8]

\draw[thick] (0,1.2)--(0,0)--(1.2,0)--(1.2,1.2);

\draw[fill](0,0) circle[radius=0.1];
\draw[fill](1.2,0) circle[radius=0.1];
\draw[fill](0,1.2) circle[radius=0.1];
\draw[fill](1.2,1.2) circle[radius=0.1];
\draw[fill](.6,2) circle[radius=0.1];

\draw[thick] (4,1.2)--(4,0);
\draw[thick] (5.2,0)--(5.2,1.2);

\draw[fill](4,0) circle[radius=0.1];
\draw[fill](5.2,0) circle[radius=0.1];
\draw[fill](4,1.2) circle[radius=0.1];
\draw[fill](5.2,1.2) circle[radius=0.1];
\draw[fill](4.6,2) circle[radius=0.1];

\draw[thick] (8.6,.8)--(8,0)--(9.2,0)--(8.6,.8)--(8.6,1.5);

\draw[fill](8,0) circle[radius=0.1];
\draw[fill](9.2,0) circle[radius=0.1];
\draw[fill](8.6,.8) circle[radius=0.1];
\draw[fill](8.6,1.5) circle[radius=0.1];
\draw[fill](8.6,2.2) circle[radius=0.1];


\end{tikzpicture}
\caption{$K_1\cup P_4$, $K_1\cup K_2\cup K_2$  and  $K_1\cup H$   \label{fig:paw}}
\end{center}\vspace{-8pt}
\end{figure}

\begin{lemma}\label{spectra2}
\begin{enumerate}
\item[{\rm(a)}] The Seidel spectrum of $K_1\cup P_4$ is $[\sqrt{5}]^2, [0]^1, [-\sqrt{5}]^2$;  $-\sqrt{5}\approx -2.2361$.
\item[{\rm(b)}]  The  Seidel spectrum of $K_1\cup K_2\cup K_2$ is $\left[\frac{1+\sqrt{17}}{2}\right]^1,[1]^2,\left[\frac{1-\sqrt{17}}{2}\right]^1, [-3]^1$;
    $\frac{1-\sqrt{17}}{2}\approx -1.5616$.
\item[{\rm(c)}] The Seidel spectrum of the $K_1\cup H$, where $H$ is the paw graph
{\begin{tikzpicture}[scale=0.4] \draw[thick](.6,.8)--(0,0)--(1.2,0)--(.6,.8)--(.6,1.5);
\draw[fill](0,0) circle[radius=0.1];
\draw[fill](1.2,0) circle[radius=0.1];
\draw[fill](.6,.8) circle[radius=0.1];
\draw[fill](.6,1.5) circle[radius=0.1];
\end{tikzpicture}}, is  \\$\left[\frac{1+\sqrt{17}}{2}\right]^1,[1]^2,\left[\frac{1-\sqrt{17}}{2}\right]^1, [-3]^1$.
\end{enumerate}
\end{lemma}

\begin{proof}
Parts (a) and (b) are found by direct computation. For (c) note that $K_1\cup H$ is obtained from $K_1\cup K_2\cup K_2$ by switching with respect to
$U=\{u\}$, where $u$ is an end vertex of one of the two edges.
\end{proof}

\section{Complete multipartite graphs}\label{comultipar}
For considering general complete multipartite graphs, we need some facts about their Seidel spectrum.
The Seidel spectrum of the complete multipartite graph $K_{p,p,\ldots, p}$ was found
in \cite{LvWeiZhao2012} by computing the characteristic polynomial of the graph. We
include here a different proof.

\begin{lemma}\label{spectra0}
The Seidel spectrum of the graph $K_{p,p,\ldots, p}$ on $n=kp$ vertices is
\[\{[2p-1]^{k-1},[-1]^{n-k}, [-n+2p-1]^1\}.\]
\end{lemma}

\begin{proof}
The  Seidel matrix of  $K_{p,p, \ldots, p}$ is $(2I_k-J_k)\otimes J_p -I_n$, where
$\otimes$ denotes the Kronecker product, and $n=kp$. The spectrum of $2I_k-J_k$ is $\{[2]^{k-1},[2-k]^1\}$, and that of $J_p$ is $\{[p]^1,[0]^{p-1}\}$. The spectrum of
their Kronecker product consists of all the products of these eigenvalues: $\{[2p]^{k-1},[0]^{(p-1)k}, [(2-k)p]^1 \}=\{[2p]^{k-1},[0]^{n-k}, [2p-n]^1 \}$. Thus the Seidel spectrum
of $K_{p,p, \ldots, p}$   is $\{[2p-1]^{k-1},[-1]^{n-k}, [-n+2p-1]^1\}$.
\end{proof}

The next lemma is mostly implicit in \cite[Corollaries 2.2 and 2.3, and eq. (3)]{WangZhaoLi2014}.
We include a partial explanation for better clarity.

\begin{lemma}\label{spectra3}
Let $p_1\ge p_2\ge \ldots\ge p_k\ge 1$. Denote the  Seidel eigenvalues of the complete $k$-partite graph $K_{p_1, p_2, \ldots, p_k}$
by $\lambda_1\ge \lambda_2\ge \ldots \ge \lambda_n$, where $n=p_1+p_2+\ldots+p_k$. Then
\begin{enumerate}
\item[(1)] $\lambda_1\ge \lambda_2\ge \ldots \ge  \lambda_{k-1}>0>\lambda_{k}=-1=\ldots=-1=\lambda_{n-1}\ge \lambda_n$.
\item[(2)] $ 2p_1-1\ge \lambda_1\ge 2p_2-1\ge \lambda_2\ge \ldots\ge 2p_{k-1}-1 \ge\lambda_{k-1}\ge 2p_k-1$,
\item[(3)] $(2-k)p_k-1\ge \lambda_n\ge 2p_k-n-1$.
\item[(4)] For $i=1, \ldots, k-1$, if $p_i>p_{i+1}$, then $2p_i-1>\lambda_i>2p_{i+1}-1$.
\end{enumerate}
\end{lemma}

\begin{proof}
The graph $K_{p_k,p_k,\ldots, p_k}$ is an induced subgraph of $K_{p_1, p_2, \ldots, p_k}$, with
Seidel eigenvalues $\mu_1\ge \mu_2\ge \ldots \ge \mu_{kp_k}$. By Lemma \ref{spectra0},
\[\mu_1=\ldots=\mu_{k-1}=2p_k-1>0.\]
By interlacing we get that
\[\lambda_i\ge \mu_i>0\]
for every $i=1, \ldots, k-1$,
and
\[(2-k)p_k-1=\mu_{kp_k}\ge \lambda_{n-kp_k+kp_k}=\lambda_n.\]
Let $S$ be the Seidel matrix of $K_{p_1,p_2,\ldots, p_k}$. It is easy to see that  in $I+S$ there are only $k$ different rows (the $i$-th of which is repeated
$p_i$
times, $i=1,\ldots, k$). Thus $\rank(I+S)\le k$, and $-1$ is an  eigenvalue of $S$ of multiplicity at least $n-k$. This completes the proof of (1) (and part of (3)).

The more detailed inequalities  on the Seidel eigenvalues in  (2) and (4) follow  from analysis of the characteristic polynomial of $S$, done in \cite[Theorem 2.2]{WangZhaoLi2014}. The polynomial
is
\[(x+1)^{n-k}\left[\prod_{i=1}^k\left(x-2p_i+1\right)+ \sum_{j=1}^k p_j\prod_{\underset{i\ne j}{i=1}}^k (x-2p_i+1),\right]\]
and thus the $n-k$ eigenvalues of $S$ other than the known  $-1$'s are the roots of
\[f(x)=\prod_{i=1}^k\left(x-2p_i+1\right)+ \sum_{j=1}^k p_j\prod_{\underset{i\ne j}{i=1}}^k (x-2p_i+1). \]
The inequality $\lambda_n\ge 2p_k-n-1$ follows from
\[0=\sum_{i=1}^n\lambda_i\le \sum_{i=1}^{k-1}(2p_i-1)+(n-k)(-1)+\lambda_n,\]
combined with $\sum_{i=1}^{k-1}p_i=n-p_k$. This completes the proof of (3).
\end{proof}

\begin{theorem}\label{main}
Let the Seidel spectrum of $G$ be equal to that of  the complete $k$-partite graph $K_{p_1, p_2, \ldots, p_k}$. Then
$G$ is a complete $k$-partite graph, up to switching.
\end{theorem}

\begin{proof}
We first show that $G$ is complete multipartite up to switching.
By using a switching with respect to the neighborhood of a single vertex, we may assume that
$G=K_1\cup D$. As the Seidel spectrum of $G$ is equal to that of $K_{p_1, p_2, \ldots, p_k}$,
$\lambda_{n-1}=-1$. If  $\mu_1\ge \ldots\ge \mu_5$ are the Seidel eigenvalues of an  induced subgraph of $G$ on $5$ vertices,  $\mu_{4}\ge \lambda_{n-5+4}=-1$ by interlacing. Hence,
by Lemma \ref{spectra2}, $G$ cannot have an induced $K_1\cup K_2\cup K_2$, or an induced
$K_1\cup P_4$, or an induced $K_1\cup  H$, where $H$ is the paw graph. Therefore $D$ cannot contain an induced $K_2\cup K_2$ or an induced $P_4$, or an induced paw.

At most one of the components of $D$ may contain an edge, since otherwise
$D$ contains an induced $K_2\cup K_2$, contrary to the assumption.
Let  $\chi(D)=t$ be the chromatic number of $D$. If $t>1$ then $D$ contains (exactly one) component $C$ with
at least one edge, and $\chi(D)=\chi(C)$.
Let $C=V_1\cup V_2\cup \ldots\cup V_t$, such that $v{\not\sim}u$ whenever $v, u\in V_i$,
and for each $i\ne j$ there is an edge with one end in $V_i$ and one end in $V_j$.
Let $v\in V_i$ and $u\in V_j$ be neighbors. If $w\in V_i$, $w\ne v$, then
$w$ has to be a neighbor of $u$ too. To see that, suppose on the contrary that $w{\not\sim}u$. Then by the connectivity of $C$, $w$ has a neighbor $z\ne u$.
As $v{\not\sim}w$, the subgraph of $D$ induced on vertices  $\{v,u,z,w\}$ is one of the following: either
an induced $K_2\cup K_2$ (if it has only the edges $vu$ and $wz$), or  an induced $P_4$ (if $u\sim z$ and $v\not\sim z$), or  an induced paw (if $v$ and $u$ are both neighbors of $z$).
This contradicts the observation above, that none of these is a possible induced subgraph of $D$. Thus every neighbor of $v\in V_i$ is also
a neighbor of all the other vertices in $V_i$. Suppose $u\in V_j$, $j\ne i$. By the same argument, each vertex in $V_j$
is a neighbor of each of the vertices in $V_i$. Hence the graph induced on $V_i\cup V_j$ is complete bipartite.
Since this holds for any  $i\ne j$,  the graph $C$ is complete $t$-partite graph. Therefore $D$, and thus $G$, consists of a complete multipartite graph and isolated vertices.

Let $U$ be the set of all isolated vertices of $K_1\cup D$. Switching with respect to $U$ yields a complete $(t+1)$-partite graph.
Hence $G$ is a complete $r$-partite graph for some $r$ (up to switching). The number of positive Seidel eigenvalues of $G$
is therefore $r-1$. On the other hand, the Seidel eigenvalues of $G$ are those of $K_{p_1, p_2, \ldots, p_k}$, hence $r-1=k-1$, implying that $G$ is a
complete $k$-partite graph, up to switching.
\end{proof}

According to the next theorem, two complete $k$-partite graphs, $k\ge 3$, can be switching equivalent only if they are isomorphic.

\begin{theorem}\label{difkpart}
Let $p_1\ge p_2\ge \ldots \ge p_k\ge 1$ and $q_1\ge q_2\ge \ldots \ge q_k\ge 1$ be  two different $k$-tuples, $k\ge 3$, such
that $\sum_{i=1}^kp_i=\sum_{i=1}^kq_i$. Then $K_{p_1, p_2, \ldots, p_k}$ and $K_{q_1, q_2,\ldots, q_k}$
are not switching-equivalent.
\end{theorem}

\begin{proof}
Let $V_i$ denote the independent set in $G=K_{p_1, p_2, \ldots, p_k}$ of size $p_i$,  $i=1,\ldots, k$.
Let $W_i$, $i=1,\ldots, k$, be the independent sets in $K_{q_1, q_2,\ldots, q_k}$.
Suppose there exists $U\subseteq \cup_{i=1}^k V_i$ such that
switching $G$ with respect to $U$ yields $G'=K_{q_1, q_2,\ldots, q_k}$. The set $U$ is not
empty, since by the assumption on the $k$-tuples $G'\ne G$.

We denote $V_{i1}=V_i\cap U$, and $V_{i2}=V_i\setminus V_{i1}$.
In $G'$, each vertex in $V_{i1}$ is connected by an edge with each vertex in $V_{i2}$, and
with each vertex in $V_{j1}$, $j\ne i$, and to no other vertices. Similarly, each vertex in $V_{i2}$ is connected also
with each vertex in $V_{j2}$, $j\ne i$.

Suppose without loss of generality that $V_{11}\ne \emptyset$. Then at most one of $V_{i2}$, $i=2, \ldots, k$,
is not empty. For if $V_{i2}, V_{j2}\ne \emptyset$, for $i, j\ge 2$, $i\ne j$, then
as $V_{11}\subseteq W_\ell$ for some $\ell$, and $V_{11}\cup V_{i2}$ is independent in $G'$,
also $V_{i2}\subseteq W_\ell$. Similarly, $V_{j2}\subseteq W_\ell$. But since there are
edges between $V_{i2}$ and $V_{j2}$, this contradicts the independence of $W_\ell$.
So suppose without loss of generality that $V_{i2}=\emptyset$ for $3\le i\le k$.
Then $V_{i1}=V_i$ for $3\le i\le k$.

Now both $V_{11}$ and $V_{k1}$ are not empty. If $V_{22}\ne \emptyset$, then $V_{11}\cup V_{22}$ is independent in $G'$
and thus contained in $W_\ell$. But then since $V_{k1}\cup V_{22}$ is independent, also $V_{k1}\subseteq W_\ell$. A
contradiction, since there is an edge between $V_{11}$ and $V_{k1}$. Hence $V_{22}$ is empty, and
$V_{21}=V_2$ is non-empty. By the independence of both $V_{k1}\cup V_{12}$ and $V_{21}\cup V_{12}$ in $G'$, and
the existence of edges between $V_{k1}$ and $V_{21}$,
we get that $V_{12}$ has also to be empty. That is, $V_{11}=V_1$, and thus $U$ consists of all the vertices of $G$,
which means that $G'=G$, contrary to the assumption that the $k$-tuples are different.
\end{proof}

The last theorem leaves out the case $k=2$.

\begin{remark}\label{Kpq=Kst}
{\rm By Lemma \ref{spectra1} all complete bipartite graphs are Seidel cospectral.
However, any two complete bipartite graphs are also switching equivalent:
Let $K_{p,q}$ and $K_{s,t}$ be two non-isomorphic complete bipartite graphs, with $p+q=s+t=n$, $\{s,t\}\ne \{p,q\}$.
Suppose $p\ge q$, $s\ge t$, and $p>s$.
Let $V_1$ and $V_2$ be the independent sets in $K_{p,q}$, $|V_1|=p$, $|V_2|=q$.
Let $U\subseteq V_1$ be any set of $s-q$ vertices
(by our assumptions, $p-t=s-q$  and $p-t>s-t\ge 0$, so $p>s-q>0$). Then after switching $K_{p,q}$ with
respect to $U$, we get $K_{s,t}$  with independent sets $W_1=U\cup V_2$ and $W_2=V_1\setminus U$.}
\end{remark}

Combining Remark \ref{Kpq=Kst} with Theorem \ref{main} we get that if a graph $G$
is cospectral with $K_{p,q}$, then $G$ is switching equivalent to  a complete bipartite graph, and therefore
to $K_{p,q}$.

\begin{theorem}
Any complete bipartite graph is $S$-determined.
\end{theorem}

In some cases a complete $k$-partite graph is  determined by its
Seidel spectrum up to switching. The following is one such case.

\begin{theorem}\label{Kpiqi}
Let $p_1> \ldots > p_l\ge  1$. The graph $K_{\underset{s_1}{\underbrace{p_1,...,p_1}},\ldots,\underset{s_l}{\underbrace{p_l,...,p_l}}}$, where $s_i\geq 3$ for every $i=1,\ldots,l$, is $S$-determined.
\end{theorem}
\begin{proof}
Let  $r_i=\sum_{j=1}^{i}s_j, i=1,2,\ldots,l$, $r_0=0$. Denote $k=r_l$. Let $\lambda_1\ge \ldots\ge \lambda_n$, where
$n=\sum_{i=1}^l s_ip_i$, be the Seidel eigenvalues of the $k$-partite graph $K_{\underset{s_1}{\underbrace{p_1,...,p_1}},\ldots,\underset{s_l}{\underbrace{p_l,...,p_l}}}$. By Lemma \ref{spectra3},
for  $i=1, \ldots, l$
\[\lambda_j=2p_i-1,\quad j=r_{i-1}+1, \ldots, r_i-1,\]
and
\[2p_i-1>\lambda_{r_i}>2p_{i+1}-1.\]

If for $q_1\ge q_2\ge \ldots \ge q_k$ the graph
$K_{q_1, q_2, \ldots, q_k}$ has the same Seidel spectrum, then by Lemma \ref{spectra3}(2), for $i=1,\ldots,l$,
\[2q_j-1\geq \lambda_j=2p_i-1 \geq 2q_{j+1}-1, j=r_{i-1}+1,\ldots,  r_{i}-1.\]
Hence $q_{j}=p_i$ for  $j=r_{i-1}+2,\ldots, r_{i}-1$.

By Lemma \ref{spectra3}(4),   $q_{r_{i-1}+1}>q_{r_{i-1}+2}$  is impossible, since otherwise
\[2q_{r_{i-1}+1}-1>\lambda_{r_{i-1}+1}>2q_{r_{i-1}+2}-1,\] contrary to $\lambda_{r_{i-1}+1}=2p_i-1=2q_{r_{i-1}+2}-1$ obtained above. Hence, $q_{j}=p_i$, for $j=r_{i-1}+1,\ldots, r_{i}-1$, $i=1, \ldots, l$.

Since $\sum_{j=1}^k q_j=\sum_{i=1}^k s_ip_i$, we have     $\sum_{i=1}^{l}q_{r_i}=\sum_{i=1}^{l}p_i$.
As $q_{r_i}\le q_{r_i-1}=p_i$  for every $i=1, \ldots, l$, the latter equality of  sums implies that
$q_{r_i}=p_i$ for every $i$.
\end{proof}

In general, however, there may be complete $k$-partite graphs that are not switching-equivalent and have
the same Seidel spectrum. Some examples of such tripartite graphs are included in the next section.

\section{Complete tripartite graphs}\label{comtripar}
It was shown in \cite{LvWeiZhao2012} that the
characteristic polynomial of the Seidel matrix of the complete tripartite graph $K_{p,q,r}$ is
\[(\lambda+1)^{n-3}(\lambda^3+\lambda^2(3-n)+\lambda(3-2n)+(4pqr-n+1)),\]
where $n=p+q+r$. Thus if $\lambda_1>\lambda_2$ are the two positive Seidel eigenvalues of $K_{p,q,r}$ and
$\lambda_n$ is its smallest Seidel eigenvalue, then $\lambda_1, \lambda_2, \lambda_n$
are the roots of the polynomial \[\lambda^3+\lambda^2(3-(p+q+r))+\lambda(3-2(p+q+r))+(4pqr-(p+q+r)+1).\]
As  the other $n-3$ eigenvalues in both graphs are all equal to $-1$, $K_{x,y,z}$ has exactly the same Seidel
spectrum as $K_{p,q,r}$  if and only if
the following two equalities hold:
\begin{align}
x+y+z&=p+q+r   \label{p+q+r}\\
xyz&=pqr ~~.\label{pqr}
\end{align}
Combined with Theorem \ref{difkpart} this implies the following.

\begin{observation}\label{dif3part}
The complete tripartite graph $K_{p,q,r}$ is $S$-determined  if and only if
the unique solution (up to permutation) to the system of equations \eqref{p+q+r} and \eqref{pqr} is
 $\{x,y,z\}=\{p,q,r\}$.
\end{observation}

\begin{example}\label{661}
The graphs $K_{6,6,1}$ and $K_{9,2,2}$ have the same Seidel spectrum, but are not switching-equivalent.
\end{example}

We mention a few more observations:
\begin{itemize}
\item Suppose one of $x,y,z$ is equal to one of $p,q,r$, say $x=p$. Then
$y$ and $z$ have to have the same sum and the same product as $q$ and $r$. But $q,r$ is the only pair of integers
with sum $q+r$ and product $qr$. Thus if $K_{x,y,z}$ has the same Seidel spectrum as $K_{p,q,r}$  and is
not switching-equivalent to $K_{p,q,r}$, then $\{x,y,z\}\cap\{p,q,r\}=\emptyset$.
\item  If $K_{x,y,z}$ has
the same Seidel spectrum as $K_{p,q,r}$, then $K_{kx,ky,kz}$ has the same spectrum as $K_{kp,kq,kr}$.
Thus if $K_{p,q,r}$ is not   $S$-determined, then for every positive
integer $k$ the graph $K_{kp,kq,kr}$ is not  $S$-determined.  The converse does not hold: $K_{10,2,2}$
has the same Seidel spectrum as $K_{8,5,1}$, but $K_{5,1,1}$ is $S$-determined, see Theorem \ref{pq1} below.
\end{itemize}

It is not hard to find infinitely many pairs of different triples $p, q, r$ and $x, y, z$ such that
$K_{p,q,r}$ and $K_{x,y,z}$ have the same Seidel spectrum. Some examples will be given below. In fact, it was shown in \cite{Schinzel1996} that for every  positive integer $m$
there
are infinitely many sets of $m$ different primitive triples sharing the same sum and the same product   (where a triple is primitive
if the greatest common divisor of its elements is $1$).

In the remainder of this section, we mention some cases when complete tripartite graphs are $S$-determined, and some cases when they are not.

\begin{theorem}\label{casesdetermined}
In the following cases the graph $K_{p,q,r}$ is $S$-determined:
\begin{enumerate}
\item[{\rm(a)}]  $p=q=r$.
\item[{\rm(b)}]   $p,q,r$ are all powers of the same prime $a$.
\item[{\rm(c)}] $\max\{p,q,r\}$ is prime.
\end{enumerate}
\end{theorem}

\begin{proof}
Part (a) is a special case of Theorem \ref{Kpiqi}. Part (b) holds since in this case if $K_{x,y,z}$ has
the same Seidel spectrum as $K_{p,q,r}$, then $xyz=pqr$ implies that each of $x$, $y$, $z$ is a power of $a$.
Since there is a unique way to write $p+q+r$ as a sum of powers of the prime $a$, the triple $x,y,z$ is
equal to $p,q,r$.

We now prove part (c). Suppose $p\ge q\ge r$ and $p$ is prime.
If $xyz=pqr$, then $p$ divides one of $x,y,z$, say $x$. If $x=kp$, $k\ge 2$, then
\[3p\ge p+q+r=x+y+z\ge kp+2\]
implies that $k=2$. But then
\[qr=2yz ~\text{  and }~ q+r=p+y+z,\]
and thus $q\ge q+r-p= y+z\ge 2z$. This in turn implies
\[2yz=qr\ge 2zr,\]
hence $y\ge r$. We get that
\[p+y\ge q+r =p+y+z>p+y,\]
a contradiction. Thus $x=p$, and therefore the triple $p,q,r$ and the triple $x,y,z$ are identical.
\end{proof}

We now consider whether slightly weaker conditions may suffice for $K_{p,q,r}$ to be $S$-determined.
Does equality of exactly two elements in the triple $p,q,r$ suffice? Does a prime in the triple, but
not the largest, suffice? Do two primes suffice? In general, the answer to each of these questions is
negative:

\begin{example}\label{prrppr}
{\rm For every positive integer $k$ ($k$ and $2k-1$ not necessarily prime) the graphs
$K_{k(2k-1), k(2k-1), 1}$ and  $K_{(2k-1)^2, k, k}$ are Seidel cospectral, and for $k>1$ they are non-isomorphic.
Thus for every $r>1$ there exists $p>r$ such that $K_{p,r,r}$ is not $S$-determined, and
for every $r\ge 1$ there exist infinitely many non-prime $p$'s such that $K_{p,p,r}$ is not $S$-determined.}
\end{example}

To point out some cases where $K_{p,p,r}$, $p>r$, is $S$-determined, we first prove an auxiliary result.

\begin{lemma}\label{p=q>r}
Let $p>r$ and let $K_{x,y,z}$ be Seidel cospectral with $K_{p,p,r}$, where $x\ge y\ge z$ and $x,y,z\notin\{p,r\}$. Then
\[x>p>y\ge z>r \text{ and } x+y<2p.\]
Let $K_{x,y,z}$ be Seidel cospectral with $K_{p,r,r}$, where $x\ge y\ge z$ and $x,y,z\notin\{p,r\}$. Then
\[p>x\ge y>r>z \text{ and } x+y<2p.\]
\end{lemma}

\begin{proof}
let $\lambda_1\ge \lambda_2$ be the two positive Seidel eigenvalue  of $K_{p,p,r}$. By Lemma \ref{spectra3}, since $p>r$, \[2p-1\ge \lambda_1\ge 2p-1> \lambda_2> 2r-1.\]
Thus $\lambda_1=2p-1$, and $\lambda_2>2r-1$. By Lemma   \ref{spectra3},
\[2x-1\ge 2p-1\ge 2y-1\ge \lambda_2\ge 2z-1.\]
Since $2x-1, 2y-1\ne 2p-1$ the first two inequalities are strict, and hence $x>p>y$.
So $x=p+a$, $y=p-b$, with $a$ and $b$ positive integers.
As
\[p+a+p-b+z=2p+r,\]
we get that $r=z+a-b$. Thus from $xyz=p^2r$ we get
\[(p+a)(p-b)z=p^2(z+a-b),\]
implying that
\[(a-b)pz-abz=(a-b)p^2,\]
or
\[(a-b)p(z-p)=abz.\]
As in the last equality the right hand side is positive, and $z\le y<p$, we must have $a<b$, and
\[x+y=p+a+p-b<2p.\]
This in turn implies that
\[z=2p+r-(x+y)>r.\]
The proof of the second claim in the lemma is similar.
\end{proof}

\begin{theorem}\label{ababa}
Let $a$ and $b$ be  different  primes. Then $K_{ab,ab,a}$ is $S$-determined if and only if the ordered pair $(a,b)\ne (2,3)$.
\end{theorem}

\begin{proof}
For the only if note that the tripartite graph $K_{6,6,2}$ is Seidel cospectral with $K_{8,3,3}$.

Now suppose $(a,b)\ne (2,3)$.   By by Lemma \ref{p=q>r} combined with  \eqref{p+q+r} and \eqref{pqr}, the graph $K_{x,y,z}$, $x\ge y\ge z$ and $x,y,z\notin\{ab,a\}$,
is Seidel cospectral with $K_{ab,ab,a}$ if and only if
the following three conditions are satisfied:
\begin{align}
x+y+z&=2ab+a    \label{2ab+a}\\
xyz&=a^3b^2   \label{a3b2}\\
x>ab&>y\ge z>a  \label{x>ab>..z>a}
\end{align}
By \eqref{a3b2} $z$ cannot be a product of at least two primes, since otherwise $xy$ is a product of at most three primes, which cannot occur in combination with $x>y\ge z$.
Hence $z$ is a prime.
In the case that $b<a$,  $K_{ab,ab,a}$ is $S$-determined since in this case there is no prime $z>a$ which divides $a^3b^2$.
If $b>a$, then $z=b$. Then $xy=a^3b$, $x>y\ge b$, $x,y\notin\{ab,a\}$, holds
only if $x=a^3$, $y=z=b$.
By \eqref{2ab+a}, $a^3+2b=2ab+a$ , so $a^3-a=2b(a-1)$ and therefore $a(a+1)=2b$. But this last equality is satisfied only if $a=2$ and $b=3$, contrary to our
assumption. Therefore $K_{ab,ab,a}$ is $S$-determined in this case also.
\end{proof}

In the next result, we consider $K_{p,p,1}$, where $p$ is a product of two different primes.

\begin{theorem}\label{abab1}
If $p=ab$, where $a>b$ are both prime, and $K_{p,p,1}$ is not $S$-determined, then
$a=2b-1$ and the only complete tripartite graph Seidel cospectral with $K_{p,p,1}$ is $K_{a^2,b,b}$.
\end{theorem}

\begin{proof}
By the assumption and Lemma \ref{p=q>r}, there exist $x>p>y\ge z>1$ such that
$K_{x,y,z}$ is Seidel cospectral with $K_{p,p,1}$. Then
$xyz=a^2b^2$. By Lemma \ref{p=q>r}, $x>ab$, and as $y,z\notin\{ab,1\}$, $x\notin \{a^2b^2\,,\, a^2b\, ,\, ab^2\}$. Thus either $x=b^2$
or $x=a^2$. Suppose $x=b^2$. Then $yz=a^2$ and, since $y,z\ne 1$, necessarily $y=z=a$. By \eqref{p+q+r} we get
that
\[b^2+2a=2ab+1,\]
and therefore
\[b^2-1=2a(b-1),\]
implying that $b+1=2a$. But this is impossible by the assumption that $a>b(>1)$.
By the same computation, in the remaining case, that $x=b^2$, necessarily $y=z=a$ and we get that $a+1=2b$.
The graphs $K_{ab,ab,1}$ and $K_{a^2,b,b}$ do have the same Seidel spectrum when $a=2b-1$ (see also Example \ref{prrppr}).
\end{proof}

In the next result, we consider the general case that one element in the triple is $1$.

\begin{theorem}\label{pq1}
$K_{p,q,1}$ is $S$-determined if $q\le 4$. For every $q>4$ there exists a positive integer $p$ such that $K_{p,q,1}$ is
not $S$-determined.
\end{theorem}

\begin{proof}
If
\begin{align}
x+y+z&=p+q+1   \label{p+q+1} \\
xyz&=pq  ~~,\label{pq}
\end{align}
where $x\ge y\ge z$ is a triple different from $p, q, 1$, then in particular $y\ge z\ge 2$, and by
Theorem \ref{casesdetermined}(c) $x$ is not prime.
 Multiplying \eqref{p+q+1} by $q$ and substituting $pq$ by $xyz$ we get
\[q(x+y+z)=xyz+q^2+q,\]
and thus $q(y+z-q-1)=x(yz-q)$. As $y\ge z\ge 2$,  $yz\ge y+z$. Thus
\begin{equation} x(yz-q)\le q(yz-q-1).  \label{xle}\end{equation}

For the first part of the theorem, suppose on the contrary that such $x\ge y\ge z$ exist for $q\le 4$.
As $y,z\ge 2$, we have $yz\ge 4\ge q$. And $yz\ne q$, since
otherwise  in \eqref{xle} we get that $0\le -q$. Thus $yz>q$ and
\[x\le q\frac{yz-q-1}{yz-q}<q.\]
But since $q\le 4$ this means that $x$ is  prime, and we get a contradiction.

For the second part of the theorem suppose first that $q>4$ is odd, then $K_{\frac{(q-1)^2}{2},q,1}$ and $K_{q\frac{q-1}{2}, \frac{q-1}{2},2}$
have the same Seidel spectrum. This covers in particular the case that $q$ is prime.

If $q>4$ and $q=ab$, where $a>1$ and $b>2$, let $p=(b-1)(ab-a-b+2)$.
Then  $x=b(ab-a-b+2)$, $y=b-1$, $z=a$ form a triplet such that $K_{x,y,z}$ has the same
Seidel spectrum as $K_{p,q,1}$.
\end{proof}

\begin{corollary}\label{cornge 3 det}
 For any intger $n\ge 3$  there exists an $S$-determined complete tripartite graph of order $n$.
\end{corollary}

By computer check, complete tripartite graphs on $n$ vertices, $n<13$ or $n=15$ or $n=18$ are $S$-determined.
 For $n=13,14,16,17$ there exist complete tripartite graphs which are not $S$-determined. For $n=13, K_{9,2,2}, K_{6,6,1}$ are Seidel cospectral.
 For $n=14, K_{8,3,3}, K_{6,6,2}$ are Seidel cospectral.
 For $n=16, K_{9,5,2}, K_{10,3,3}$ are Seidel cospectral.
For $n=17, K_{9,4,4}, K_{8,6,3}$ are Seidel cospectral.


\begin{remark}{\rm
A natural question is for which integers $n$ there exist complete tripartite graphs on $n$
vertices that are not $S$-determined. This question was settled in an Undergraduate Research Project at
the Technion by Ramy Masalha and Eli Bogdanov \cite{MasalhaBogdanov2018}. They showed that for $n=7k-\alpha$, $\alpha\in \{1,2,3,4,5,6,7\}$, the
following two triples have the same sum  and the same product:
\[ \alpha, 3k,4k-2\alpha ~\text{ and }~ 2\alpha, k, 6k-3\alpha.\]
For $n\notin\{22,24,30,36,42\}$ these are two different triples. For the remaining values of $n$ they found the following pairs
of triples:
\begin{align*}
n=22:&\quad 9,8,5 \quad\text{and}\quad 10,6,6\\
n=24:&\quad 12,10,2 \quad\text{and}\quad 16,5,3\\
n=30:&\quad 20,7,3 \quad\text{and}\quad 21,5,4\\
n=36:&\quad 21,13,2 \quad\text{and} \quad 26,7,3\\
n=42:&\quad 24,16,2  \quad\text{and}\quad 32,6,4
\end{align*}
Thus for every $n\ge 13$,  other than $n=15$ and $n=18$, there exists a complete tripartite graph on $n$ vertices that is
not $S$-determined.
}\end{remark}

{\bf Acknowledgements}
 The authors would like to thank the anonymous referee for suggestions and comments  on the early version of this paper.

\end{document}